\documentclass[14pt]{amsart}

\usepackage[cp1251]{inputenc}
\usepackage[english]{babel}
\usepackage{graphicx}%
\usepackage{multirow}%
\usepackage{amsmath,amssymb,amsfonts}%
\usepackage{amsthm}%
\usepackage{mathrsfs}%
\usepackage[title]{appendix}%
\usepackage{xcolor}%
\usepackage[hidelinks]{hyperref}

\theoremstyle{thmstyleone}%
\newtheorem{theorem}{Theorem}
\newtheorem{proposition}[theorem]{Proposition}%

\begin{document}

\title[A counterexample to the Karvatskyi--Pratsiovytyi conjecture]{A counterexample to the Karvatskyi--Pratsiovytyi conjecture concerning the achievement set of an intermediate series}

\author{Mykola Moroz}
\address{\emph{M.Moroz}: Department of Dynamical Systems and Fractal Analysis, Institute of Mathematics of NAS of Ukraine, Tereschenkivska 3, 01024 Kyiv, Ukraine}

\begin{abstract}
We found a counterexample to the conjecture of Karvatskyi and Pratsiovytyi concerning the topological type of the achievement set of an intermediate series (Proceedings of the International Geometry Center, 2023. \url{https://doi.org/10.15673/pigc.v16i3.2519}). This conjecture is based on an analogy with the squeeze theorem from calculus. We also proposed an improved version of the conjecture, which this counterexample does not refute.
\end{abstract}



\maketitle

\section{Introduction}\label{sec1}
The \emph{achievement set} of a positive and absolutely summable sequence $(u_n)_{n=1}^\infty$ is defined as the set
$$E(u_n)=\left\{\sum_{n=1}^\infty \varepsilon_n u_n,~\varepsilon_n\in\{0,1\}\right\}=\left\{\sum_{n\in A}^\infty u_n,~A\subset\mathbb{N}\right\}.$$
It is well known \cite{GN1988,NS2000} that the achievement set $E(u_n)$ is one of the following three types:
\begin{itemize}
\item a finite union of closed bounded intervals;
\item a Cantor-type set (homeomorphic to the Cantor set, which is the achievement set of the sequence $(u_n)_{n=1}^\infty$, with $u_{n}=\frac{2}{3^n}$);
\item a Cantorval (homeomorphic to the Guthrie--Nymann achievement set, which is the achievement set of the sequence $(u_n)_{n=1}^\infty$, with $u_{2k}=\frac{2}{4^k}$ and $u_{2k-1}=\frac{3}{4^k}$).
\end{itemize}

It is often difficult to determine the topological type of the achievement set of a given sequence. If both inequalities $u_n>r_n^u=\sum_{i=n+1}^\infty u_i$ and $u_n\leq r_n^u$ hold infinitely often, this problem is non-trivial even for relatively simple sequences (e.g., multigeometric sequences \cite{BBSzF2015,BPr-W2017,BFSz2014}).

The conditions under which topological types of the achievement sets coincide could help partially solve this problem. In \cite{PrK2023}, Pratsiovytyi and Karvatskyi expressed a conjecture which is based on an analogy with the squeeze theorem from calculus.

\textbf{The Karvatskyi--Pratsiovytyi conjecture \cite{PrK2023}.} 
\textit{Let $\sum_{n=1}^{\infty} a_n$ and $\sum_{n=1}^{\infty} b_n$ be two convergent positive series with the same topological type of achievement sets of the sequences $(a_n)_{n=1}^\infty$ and $(b_n)_{n=1}^\infty$. If the sequences $(a_n)_{n=1}^\infty$, $(b_n)_{n=1}^\infty$, and $(c_n)_{n=1}^\infty$ satisfy the conditions
	\begin{equation}\label{1}
	a_n \leq c_n \leq b_n\text{~~~~~~and~~~~~~}	\left[
		\begin{array}{l}
			\displaystyle b_n \leq r^{a}_n=\sum_{i=n+1}^{\infty}{a_i}, \\
			\displaystyle r^{b}_n=\sum_{i=n+1}^{\infty}{b_i} < a_n \\
		\end{array}
		\right.
	\end{equation}
for all $n \in\mathbb{N}$, then $E(a_n)$, $E(b_n)$, and $E(c_n)$ have the same topological type.}

The conditions \eqref{1} are sufficiently strong  and ensure the following properties:
\begin{itemize}
\item for the sequences $(a_n)_{n=1}^\infty$, $(b_n)_{n=1}^\infty$, and $(c_n)_{n=1}^\infty$, the inequalities of the form $u_n\leq r_n^u$ or $r_n^u<u_n$ hold for the same values of $n$.
\item the difference $b_n-a_n$ is bounded from above by $\max\left\{\left|a_n-r_n^a\right|,\left|b_n-r_n^b\right|\right\}$, because $a_n \leq b_n\leq r_n^a$ or $r_n^b<a_n \leq b_n$ for all $n$. This property is slightly stronger than boundedness of the difference $b_n-a_n$, which follows from the convergence of both series $\sum_{n=1}^{\infty} a_n$ and $\sum_{n=1}^{\infty} b_n$.
\end{itemize}
 
However, these conditions are insufficient.

\section{Counterexample}\label{sec2}

Consider the following multigeometric sequences $(a_n)_{n=1}^\infty$, $(b_n)_{n=1}^\infty$, and $(c_n)_{n=1}^\infty$:
	\begin{align*}
	a_{2k-1}=&a_{2k}=\frac{\alpha}{4^k}, ~~\alpha=1.95;\\
	b_{2k-1}=&\frac{4\beta}{4^k},~b_{2k}=\frac{3\beta}{4^k},~~\beta=0.8; \\
	c_{2k-1}=&\frac{3}{4^k},~c_{2k}=\frac{2}{4^k}.
	\end{align*}

\begin{proposition}
	The sequences $(a_n)_{n=1}^\infty$, $(b_n)_{n=1}^\infty$, and $(c_n)_{n=1}^\infty$ satisfy the conditions \eqref{1} for all $n \in\mathbb{N}$.
\end{proposition}

\begin{proof}
It is obvious that $a_n \leq c_n \leq b_n$ for all $n\in\mathbb{N}$. 

If $n=2k-1$, then $b_n=b_{2k-1}=\frac{4\beta}{4^k}$ and $r_n^{a}=r_{2k-1}^{a}=\frac{5\alpha}{3\cdot 4^k}$. In this case, $b_n \leq r_n^{a}$.

If $n=2k$, then $a_n=a_{2k}=\frac{\alpha}{4^k}$ and $r_n^{b}=r_{2k}^{b}=\frac{7\beta}{3\cdot 4^k}$. In this case, $r_n^{b}<b_n$.
\end{proof}

It is known \cite{BBSzF2015} that the achievement set of multigeometric sequence $$\left(k_0,k_1,\ldots,k_m,k_0q,k_1q,\ldots,k_mq,k_0q^2,k_1q^2,\ldots,k_mq^2,\ldots\right)$$ is self-similar set of the form
\begin{equation*}
	K(\Sigma;q)=\left\{\sum_{n=0}^\infty d_nq^n\colon d_n\in\Sigma\right\},
\end{equation*}
where $\Sigma =\left\{\sum_{n=0}^{m}k_{n}\varepsilon_{n}\colon\varepsilon_{n}\in \{0,1\}\right\}$.

Consequently
\begin{align*}
E(a_n)&=	K\left(\Sigma(a_n);\frac{1}{4}\right),~\text{where}~\Sigma(a_n)=\left\{0,\frac{\alpha}{4},\frac{\alpha}{2}\right\},\\
E(b_n)&=K\left(\Sigma(b_n);\frac{1}{4}\right),~\text{where}~\Sigma(b_n)=\left\{0,\frac{3\beta}{4},\beta,\frac{7\beta}{4}\right\},\\
E(c_n)&=K\left(\Sigma(c_n);\frac{1}{4}\right),~\text{where}~\Sigma(c_n)=\left\{0,\frac{2}{4},\frac{3}{4},\frac{5}{4}\right\}.
\end{align*}

In \cite[Theorem 1.3(e)]{BBSzF2015} proved that $K(\Sigma;q)$ is a Cantor-type set of zero Lebesgue measure if $q<\frac1{|\Sigma|}$ or, more generally, if $q^n<\frac1{|\Sigma_n|}$ for some
$n\in\mathbb{N}$, where $\Sigma_n=\left\{\sum_{k=0}^{n-1}d_kq^k\colon d_k\in\Sigma\right\}$.


\begin{proposition}
The achievement set $E(a_n)$ is the Cantor-type set.
\end{proposition}

This proposition follows from Theorem 1.3(e) in \cite{BBSzF2015}, since $\left|\Sigma(a_n)\right|=3$.

\begin{proposition}
The achievement set $E(b_n)$ is the Cantor-type set.
\end{proposition}

Since $\Sigma_2(b_n)=\left\{0,\frac{3\beta}{16},\frac{\beta}{4},\frac{7\beta}{16},\frac{3\beta}{4},\frac{15\beta}{16},\beta,\frac{19\beta}{16},\frac{5\beta}{4},\frac{23\beta}{16},\frac{7\beta}{4},\frac{31\beta}{16},2\beta,\frac{35\beta}{16}\right\}$, 
$\left|\Sigma_2(b_n)\right|=14$, and $\left(\frac{1}{4}\right)^2<\frac{1}{\left|\Sigma_2(b_n)\right|}$, therefore, this proposition follows from Theorem~1.3(e) in \cite{BBSzF2015}.

However, the achievement set $E(b_n)$ is Cantorval (the Guthrie--Nymann achievement set \cite{GN1988}). Therefore, the Karvatskyi--Pratsiovytyi conjecture is false.

\section{The improved version of the Karvatskyi--Pratsiovytyi conjecture}\label{sec11}

Under the conditions of the Karvatskyi--Pratsiovytyi conjecture, it is possible that $\lim_{n\to\infty}\frac{b_n}{a_n}\not=1$. This played a crucial role in constructing the counterexample to the conjecture. Therefore, we suggest adding a condition $\lim_{n\to\infty}\frac{b_n}{a_n}\not=1$ that eliminates this disadvantage. So we get

\textbf{The improved Karvatskyi--Pratsiovytyi conjecture}
\textit{Let $\sum_{n=1}^{\infty} a_n$ and $\sum_{n=1}^{\infty} b_n$ be two convergent positive series with the same topological type of achievement sets of the sequences $(a_n)_{n=1}^\infty$ and $(b_n)_{n=1}^\infty$. If the sequences $(a_n)_{n=1}^\infty$, $(b_n)_{n=1}^\infty$, and $(c_n)_{n=1}^\infty$ satisfy the conditions
	\begin{equation}\label{2}
		\lim_{n\to\infty}\frac{b_n}{a_n}=1,~~~a_n \leq c_n \leq b_n,\text{~~~~and~~~~}	\left[
		\begin{array}{l}
			\displaystyle b_n \leq r^{a}_n=\sum_{i=n+1}^{\infty}{a_i}, \\
			\displaystyle r^{b}_n=\sum_{i=n+1}^{\infty}{b_i} < a_n \\
		\end{array}
		\right.
	\end{equation}
	for all $n \in\mathbb{N}$, then $E(a_n)$, $E(b_n)$, and $E(c_n)$ have the same topological type.}

\section*{Acknowledgements}


\begin{thebibliography}{9}
	
	
\bibitem{BBSzF2015}
T. Banakh, A. Bartoszewicz, E. Szymonik, and M. Filipczak, \emph{Topological and measure properties of some self-similar sets}, Topological Methods in Nonlinear Analysis. \textbf{46}:2  (2015), 1013--1028. \url{https://doi.org/10.12775/TMNA.2015.075}

\bibitem{BPr-W2017}
M. Banakiewicz and Fr. Prus-Wi\'{s}niowski, \emph{M-Cantorvals of Ferens type}, Mathematica Slovaca. \textbf{67}:4 (2017), 907--918. \url{https://doi.org/10.1515/ms-2017-0019}

\bibitem{BFSz2014}
A. Bartoszewicz, M. Filipczak, and E. Szymonik, \emph{Multigeometric sequences and Cantorvals}, Central European Journal of Mathematics. \textbf{12}:7 (2014), 1000--1007. \url{https://doi.org/10.2478/s11533-013-0396-4}

\bibitem{GN1988}
J. A. Guthrie and J. E. Nymann, \emph{The topological structure of the set of subsums of an infinite series}, Colloquium Mathematicae. \textbf{55}:2 (1988), 323--327. \url{https://doi.org/10.4064/CM-55-2-323-327}

\bibitem{Kakeya1914}
S. Kakeya, \emph{On the partial sums of an infinite series}, T\^{o}hoku Sci. Rep. \textbf{3} (1914), 159--164.

\bibitem{NS2000}
J. E. Nymann and R. A. S\'{a}enz, \emph{On the paper of Guthrie and Nymann on subsums of infinite series}, Colloquium Mathematicae. \textbf{83}:1 (2000), 1--4. \url{https://doi.org/10.4064/cm-83-1-1-4}


\bibitem{PrK2023}
M. Pratsiovytyi and D. Karvatskyi, \emph{Cantorvals as sets of subsums for a series related with trigonometric functions}, Proceedings of the International Geometry Center. \textbf{16}:3-4 (2023), 262--271. \url{https://doi.org/10.15673/pigc.v16i3.2519}

\end{thebibliography}
\end{document}